%% file: bare_jrnl1.tex
\documentclass[journal,12pt,onecolumn]{IEEEtran}
\input{rrp-macro-file}

\ifCLASSINFOpdf
\else
\fi
\begin{document}
%
\title{Phase Retrieval by Alternating Minimization with Random Initialization}
%
%
%

\author{Teng Zhang
\thanks{T. Zhang was with the Department
of Mathematics, University of Central Florida, Orlando,
FL, 32765 USA e-mail: teng.zhang@ucf.edu.}
}

\maketitle

\begin{abstract}
We consider a phase retrieval problem, where the goal is to reconstruct a $n$-dimensional complex vector from its phaseless
scalar products with $m$ sensing vectors, independently sampled from complex normal distributions. We show that, with a random initialization, the classical algorithm of alternating minimization succeeds with high probability as $n,m\rightarrow\infty$ when ${m}/{\log^3m}\geq Mn^{3/2}\log^{1/2}n$ for some $M>0$. This is a step toward proving the conjecture in \cite{Waldspurger2016}, which conjectures that the algorithm succeeds when $m=O(n)$. The analysis depends on an
 approach that enables the decoupling of the dependency between the algorithmic iterates and the sensing vectors.
\end{abstract}


%
\IEEEpeerreviewmaketitle

\section{Introduction}
%
%
%
%
\IEEEPARstart{T}{his} article concerns the phase retrieval problem as follows: let $\bz\in\bbC^n$ be an unknown vector, and given $m$ known sensing vectors $\{\ba_i\}_{i=1}^m\in\bbC^n$, we have the observations
\[
y_i=|\ba_i^*\bz|, i=1,2,\cdots,m.
\]
Then can we reconstruct $\bz$ from the observations $\{y_i\}_{i=1}^m$? In this work, we assume that the sensing vectors $\{\ba_i\}_{i=1}^m$ are sampled from a complex normal distribution $CN(0,\bI)$. That is, both its real component and its imaginary component follows from a real Gaussian distribution of $N(0,\bI/2)$.

This problem is motivated from the applications in imaging science, and we refer interested readers to ~\cite{Shechtman2015,Candes7029630} for more
detailed discussions on the background in engineering and additional applications in other areas of sciences and engineering.

Because of the practical ubiquity of the phase retrieval problem, many algorithms and theoretical analysis have been developed for this problem. For example, an interesting recent approach is based on convex relaxation~\cite{Chai2011,Candes_PhaseLift,Waldspurger2015}, that replaces the non-convex measurements by convex measurements through relaxation. Since the associated optimization problem is convex, it has properties such as convergence to the global minimizer, and it has been shown that under some assumptions on the sensing vectors, this method recovers the correct $\bz$~\cite{Candes2014,Gross2015}. However, since these algorithms involve semidefinite programming for $n\times n$ positive semidefinite matrices, the computational cost could be prohibitive when $n$ is large. Recently, several works \cite{pmlr-v54-bahmani17a,Goldstein2016,Hand2016,Hand20162,NIPS2018_8082} proposed and analyzed an alternate convex method that uses linear programming instead of  semidefinite programming, which is more computationally efficient, but the program itself requires an ``anchor vector'', which needs to be a good approximate estimation of $\bz$.

Another line of works are based on Wirtinger flows, i.e., gradient flow in the complex setting~\cite{Candes7029630,NIPS2015_5743,Zhang:2016:PNP:3045390.3045499,NIPS2016_6319,cai2016,NIPS2016_6061,Soltanolkotabi2017,Chen2018}. Some theoretical justifications are also provided \cite{Candes7029630,Soltanolkotabi2017}. However, since the objective functions are nonconvex, many of these algorithms require careful initializations, which are usually only justified when the measurement vectors follow a very specific model. In addition, there are technical issues in implementation such as choosing step sizes, which makes the implementation slightly more complicated. In particular, most theoretical analysis simply assume sufficiently small step sizes, which does give a clear guidance to practice.

The most widely used method is perhaps the alternate minimization (Gerchberg-Saxton) algorithm and its variants~\cite{Gerchberg72,Fienup78,Fienup82}, that is based on alternating projections onto nonconvex sets~\cite{Bauschke03}. As a result, in some literature it is also called the alternating projection method \cite{Waldspurger2016}.  This method is very simple to implement and is parameter-free. However, since it is a nonconvex algorithm, its properties such as convergence are only partially known. Netrapalli et al. \cite{Netrapalli7130654} studied a resampled version of this algorithm and established its convergence as the number of measurements $m$ goes to infinity when the measurement vectors are independent standard complex normal vectors. Marchesini et al. \cite{Marchesini2016815} studied and demonstrated the necessary and sufficient conditions for the local convergence of this algorithm. Recently, Waldspurger \cite{Waldspurger2016} showed that when $m \geq Cn$ for sufficiently large $C$, the alternating minimization algorithm succeeds with high
probability, provided that the algorithm is carefully initialized. This work also conjectured that the alternate minimizations algorithm with random initialization succeeds with $m\geq Cn$.

One particular difficulty in the analysis of the alternating minimization algorithm is the stationary points. Currently, most papers on nonconvex algorithms depend on the analysis showing that all (attractive) stationary points of the algorithm are well-behaved in the sense that it is the desired solution, or close to the desired solution, for example, \cite{7541725}. Then the standard algorithm such as gradient descent algorithm or trust-region method can be applied to the problem to obtain the stationary point. However, as pointed out in \cite{Waldspurger2016},  in the regime $m = O(n)$, the alternating minimization algorithm has attractive stationary points that are not the desired solution. While empirically these undesired stationary points are not obstacles for the success of the algorithm since their attraction basins seem small, but it prevents us from applying the common approach of analyzing stationary points.

Recently,  \cite{Zhang2017} shows that the algorithm improves the correlation between the estimator and the truth in each iteration with high probability. Based on this observation, it shows that a resampled version of the alternating minimization algorithm converges to the solution with high probability when $m=O(n\log^5n)$. However, this approach can not be applied to analyze the alternating minimization algorithm directly, since the estimator at the $k$-th iteration is correlated with the sensing vectors. As a result, to analyze the non-resampled version,  one needs to find a way to decouple the estimator at the $k$-th iteration and the sensing vectors.

The contribution of this work is to show that the alternating minimization algorithm with random initialization succeeds  with high probability when $m/\log^3m>Mn^{1.5}\log^{0.5}n$. While it still does not match the conjecture of $m=O(n)$, it is the best result on the classic algorithm with random initialization and without any resampling or construction of a good initialization yet. Compared with \cite{Zhang2017}, the novelty in this analysis is the decoupling of the sensing vectors and the estimator at the $k$-th iteration. The approach fixes the first $k-1$ algorithmic iterates are fixed and analyzes the conditional distribution of the sensing vectors. This approach, inspired by the analysis of LASSO in \cite{6069859}, is the main technical contribution of this work. In spirit, this contribution is very similar to  leave-one-out approach that also enables decoupling in \cite{Chen2018}, and based on their leave-one-out approach, they show that an algorithm for the phase retrieval converges linearly. However, the analyzed algorithm is very different and their work assumes that the sensing vectors and the $\bz$ are real-valued. In addition, it seems more difficult to apply the leave-one-out approach here, as the iterations is a little bit more complicated. 

The paper is organized as follows. Section~\ref{sec:main} presents the algorithm and the main results of the paper, Theorem~\ref{thm:main}. The proof of Theorem~\ref{thm:main} is given in Section~\ref{sec:proof},  where the main proof of Theorem~\ref{thm:main} is given in Section~\ref{sec:proof_main}, the proof of the main lemmas are given in Section~\ref{sec:lemma}, and the auxiliary lemmas and their proofs are given in Section~\ref{sec:auxillary}. 

\subsection{Notations}
For any $z\in\bbC$, $|z|$ represents the modulus of $z$. We use $\Sp(\ba_1,\cdots,\ba_n)$ to represent the subspace spanned  by $\ba_1,\cdots,\ba_n$, i.e., the set $\{\bx: \bx+\sum_{i=1}^nc_i\ba_i,\,\,\,\text{for}\,\, c_1,\cdots,c_n\in\bbC\}$. Note that here the subspace is slightly different from the standard subspace, where the coefficient of each vector is a complex number.  We use $P_L$ to denote the projection onto the subspace $L$: $P_L(\bz)$ represents the nearest point on $L$ to $\bz$.

For any $z\in\bbC$, $\phase(z)=z/|z|$ is the phase of $z$. For any vector $\bz=(z_1,\cdots,z_m)$, $\phase(\bz)$ is the phases for each elements:
\[
\phase(\bz)=(\phase(z_1), \cdots, \phase(z_m)).
\]
We use $\odot$ to denote the pointwise product between the phase of the first vector and the modulus of the second vector. That is,
\[
(\bw\odot \by)_i=\frac{w_i}{|w_i|}|y_i|.
\]
For any vector $\bz\in\bbC^m$, $\|\bz\|$ represents its Euclidean norm: $\|\bz\|=\sqrt{\sum_{i=1}^m|z_i|^2}$, and its $1$-norm and $\infty$-norm are defined by $\|\bz\|_1=\sum_{i=1}^m|z_i|$ and $\|\bz\|_\infty=\max_{1\leq i\leq m}|z_i|$.

\subsection{Algorithm and Main result}\label{sec:main}
The alternating minimization method is one of the earliest methods that was introduced for phase retrieval problems~\cite{Gerchberg72,Fienup78,Fienup82}, and it is based on alternating projections onto nonconvex sets~\cite{Bauschke03}. Let  $\bA\in\bbC^{m\times n}$ be a matrix with columns given by $\ba_1,\ba_2,\cdots,\ba_m$, then its goal is to find a vector in $\bbC^m$ such that it lies in both the subspace $L=\range(\bA)\in\bbC^m$ and the set of correct amplitude $\mathcal{A}=\{\bw\in\bbC^m: |\bw_i|=y_i\}$. For this purpose, the algorithm  picks an initial guess $\bx^{(1)}$ in $\bbC^n$ and alternatively project $\bA\bx^{(1)}$ on the both sets.  Let $\bw^{(k)}=\bA\bx^{(k)}$ for all $k\geq 1$,  then the projections $P_L, P_\mathcal{A}: \bbC^m\rightarrow\bbC^m$ can be defined by
\[
P_L(\bw)=\bA(\bA^*\bA)^{-1}\bA^*\bw,\,\,\,[P_\mathcal{A}(\bw)]_i=y_i\frac{\bw_i}{|\bw_i|},
\]
and the alternating minimization algorithm is given by applying the operator $P_LP_{\mathcal{A}}$ recursively to the vector $\bw^{(1)}$, i.e., \begin{equation}\label{eq:alternating_algorithm}\bw^{(k+1)}=P_LP_{\mathcal{A}}\bw^{(k)}.\end{equation} Then the estimator of $\bx$ at the $k$-th iteration is obtained by solving $\bx^{(k)}=\bA\bw^{(k)}$.

This algorithm has been studied in \cite{Waldspurger2016} and Theorem 2 in \cite{Waldspurger2016}  shows the convergence of the algorithm if $m>Mn$ and if there is a good initialization. In addition, it conjectures that random initialization	also succeed in this setting.  In this article, we prove that this conjecture holds when $m/\log^2m>Mn^{1.5}\log^{0.5} n$ for some $M>0$. The rigorous statement is as follows:
\begin{thm}\label{thm:main}
Assuming that the sensing vectors $\{\ba_i\}_{i=1}^m$ are i.i.d. sampled from the complex normal distribution $CN(0,\bI)$, then there exists $M>0$ such that if ${m}/{\log^3m}\geq Mn^{3/2}\log^{1/2}n$, then the alternating projection algorithm with random initialization (obtained from a uniform distribution on the sphere of $\bbC^n$) succeeds almost surely in the sense that  as $n,m\rightarrow\infty$,
\begin{equation}\label{eq:main}
\Pr\left(\lim_{k\rightarrow\infty}\inf_{\psi\in \reals}\|e^{i\psi}\bx^{(k)}-\bz\|=0\right)\rightarrow 1.
\end{equation}
\end{thm}

In the proof, for simplicity when we talk about a ``random unit vector in $\bbC^m$/subspace $L$'' , we implicitly assume that it is sampled from the uniform distribution on the unit sphere in $\bbC^m$ or subspace $L$.  The constants $c, C$ are used to represent a constant that is independent of $m$ and $n$, and it is used to represent different constants in different equations. In addition, since the theorem focus on the setting when $n$ and $m$ both large, we write down inequalities under this assumption. For example, we may write $\log^3n<n$ even though it only holds for large $n$.

\section{Proof of Theorem~\ref{thm:main}}\label{sec:proof}
In the proof, we will first present reduced form of the statement of Theorem~\ref{thm:main} in Section~\ref{sec:equivalent}, and then present the proof of this reduced statement in Section~\ref{sec:proof_main}. The proof of the main lemmas are given in Section~\ref{sec:lemma}, and the auxiliary lemmas (which are mostly generic results on measure concentration) and their proofs are given in Section~\ref{sec:auxillary}.
\subsection{An Equivalent form of Theorem~\ref{thm:main}}\label{sec:equivalent}
In this section we first introduce a few assumptions on $\bA$ and some modification of the algorithm, which does not impact the performance of the algorithm but would simplify the proof later.

First, we investigate the performance of the same algorithm if the sensing matrix $\bA$, the underlying signal $\bz$ and the initialization $\bx^{(1)}$ are replaced by $\tilde{\bA}=\bA\bD$, $\tilde{\bz}=\bD^{-1}\bz$, and $\tilde{\bx}^{(1)}=\bD^{-1}{\bx}^{(1)}$ respectively, for some $\bD\in\bbC^{n\times n}$. Then $\bw^{(1)}$ and $\by$ are unchanged, and $\mathrm{range}(\tilde{\bA})=\mathrm{range}({\bA})$, which means that the updates in \eqref{eq:alternating_algorithm} is unchanged, and the estimators between these two settings have the connection of $\tilde{\bx}^{(k)}=\bD^{-1}\bx^{(k)}$. As a result, $\|e^{i\psi}\tilde{\bx}^{(k)}-\tilde{\bz}\|\rightarrow 0$  if and only if $\|e^{i\psi}\bx^{(k)}-\bz\|\rightarrow 0$.  For the rest of the proof, we will analyze an equivalent problem, where $\bD=(\bA^*\bA)^{-1/2}$ and $\bA$ is replaced with $\bA(\bA^*\bA)^{-1/2}$, the projection matrix to the subspace $L$.  

Second, WLOG we assume that $\|\bz\|=1$ (which implies that $\|\by\|=1$ because $\bA$ is a projection matrix) and we normalize $\bw$ in the update formula \eqref{eq:alternating_algorithm}:
\begin{equation}\label{eq:alternate_minimization}
\bw^{(k+1)}=\frac{P_LP_{\mathcal{A}}\bw^{(k)}}{\|P_LP_{\mathcal{A}}\bw^{(k)}\|}.
\end{equation}
Compared with the original form \eqref{eq:alternating_algorithm},  $\bw^{(k)}$ is normalized to a unit vector in each iteration. Since the operator $P_{\mathcal{A}}$ is invariant to the scaling, and $\bw^{(k+1)}$ depends on $\bw^{(k)}$ through $P_{\mathcal{A}}\bw^{(k)}$, the alternating minimization algorithm with normalization \eqref{eq:alternate_minimization} is equivalent to the standard version \eqref{eq:alternating_algorithm} with a ``correct'' scaling, and it is relatively straightforward to verify that Theorem~\ref{thm:main} holds for \eqref{eq:alternate_minimization} if and only if it holds for \eqref{eq:alternating_algorithm}.

Since $\{\ba_i\}_{i=1}^n$ are i.i.d. sampled from $CN(0,\bI_{m\times m})$, $L$ is a random $n$-dimensional subspace in $\bbC^m$.
Combining the analysis above, to prove Theorem~\ref{thm:main}, we will address the following equivalent problem:
\begin{itemize}
    \item Choose a unit vector $\bz\in\bbC^n$ and a random $n$-dimensional subspace $L$ in $\bbC^m$, and a random unit vector on $L$, denote it by $\bw^{(1)}$. Let $\by=|\Proj_L^*\bz|$, where $\Proj_L\in\bbC^{m\times n}$ represents a random projection matrix to $L$ (there are many choices of $\Proj_L$: for any unitary matrix $\bU\in\bbC^{n\times n}$,  $\Proj_L\bU$ is another projection matrix to $L$, and we randomly choose one).
    \item The iterative update formula is given by \begin{equation}\bw^{(k+1)}=\frac{P_{L}[\bw^{(k)}\odot \by]}{\|P_{L}[\bw^{(k)}\odot \by]\|},\label{eq:alternate_minimization2}\end{equation} and $\bx^{(k)}=\Proj_L^*\bw^{(k)}$.
    \item Goal: prove \eqref{eq:main}.
  \end{itemize}

\subsection{Main Proof}\label{sec:proof_main}
In the proof, we first define a set of orthogonal unit vectors in $\bbC^m$:
\begin{align*}
\bu_0&={\Proj_L^*\bz},\,\,\text{(note that $\|\bu_0\|=1$ since $\|\bz\|=1$)}\\
\bu_{k}&=\frac{\bw^{(k)}-\sum_{i=0}^{k-1}\bu_i\bu_i^*\bw^{(k)}}{\|\bw^{(k)}-\sum_{i=0}^{k-1}\bu_i\bu_i^*\bw^{(k)}\|},\,\,\text{for all $1\leq k\leq d$},
\end{align*}
where $d=C_d\log n$ with constant $C_d=\frac{1}{2\log(\frac{C_f+3}{4})}+1$. $C_f$ will be defined later in Lemma~\ref{lemma:expectation}, and it does not depend on $n$ or $m$.

Since $d<m$, $\{\bu_i\}_{i=0}^d$ is a set of $d+1$ orthogonal vectors in $\bbC^m$. By definition, $\bw^{(k)}\in \Sp(\bu_0,\bu_1,\cdots, \bu_{k})$ for any $1\leq k\leq d$ and $\bw^{(k)}$  can be written as $\bu_i$: \[
\bw^{(k)}=\sum_{i=0}^{k}c_i^{(k)}\bu_i.
\]
By writing $P_{L}[\bw^{(k)}\odot \by]$ in the basis of $\bu_0,\cdots,\bu_{k+1}$ as $P_{L}[\bw^{(k)}\odot \by]=\sum_{i=0}^{(k+1)}\tilde{c}_{k+1}^{(i)}\bu_i$, the update formula \eqref{eq:alternate_minimization2} can be then rewritten as the update of $\{c_i^{(k)}\}_{i=0}^k$ as follows:
\begin{align}\label{eq:alternate_minimization3}
\tilde{c}_i^{(k+1)}&=\bu_i^*\left[\left(\sum_{i=0}^{k}c_i^{(k)}\bu_i\right)\odot \bu_0\right],\,\,\,0\leq i\leq k\\\label{eq:alternate_minimization4}
\tilde{c}_{k+1}^{(k+1)}&
=\left\|P_{L}\left[\left(\sum_{i=0}^{k}c_i^{(k)}\bu_i\right)\odot \bu_0-\sum_{i=0}^k\tilde{c}_i^{(k+1)}\bu_i\right]\right\|\\
c_i^{(k+1)}&=\frac{\tilde{c}_i^{(k+1)}}{\sqrt{\sum_{i=0}^{k+1}\tilde{c}_i^{(k+1)\,2}}},\,\,\,0\leq i\leq k+1.\label{eq:alternate_minimization5}
\end{align}
While \eqref{eq:alternate_minimization4} seems complicated, this explicit formula will not be used later in the proof. Instead, the estimations
\begin{equation}\label{eq:alternate_minimization6}
0\leq \tilde{c}_{k+1}^{(k+1)}\leq \sqrt{1-\sum_{i=0}^k|\tilde{c}_i^{(k+1)}|^2}
\end{equation}
and \eqref{eq::approximate_basis2} (will be presented later) are sufficient, where the second inequality of \eqref{eq:alternate_minimization6}
 follows from the fact that $\sum_{i=0}^{k+1}|\tilde{c}_i^{(k+1)}|^2=\|P_{L}[\bw^{(k)}\odot \by]\|^2\leq \|\bw^{(k)}\odot \by\|^2=\|\by\|^2=1$.


The outline of the proof is as follows: first, we show that $\bu_i$ can be well approximately by random vectors $\bv_i$ from $CN(0,\bI/m)$ in Lemma~\ref{lemma:approximate_basis}. This step decouples the dependency between the sensing vectors and the estimations at the $k$-th iteration. Second, we investigate that the approximate dynamic of  $\{{c}_{k}^{(i)}\}_{i=0}^k$ defined in \eqref{eq:alternate_minimization3} - \eqref{eq:alternate_minimization5}, by replacing $\bu_i$  with $\bv_i$ in Lemma~\ref{lemma:iteration_v} and~\ref{lemma:expectation}. Third, we obtain the dynamic of $\{{c}_{k}^{(i)}\}_{i=0}^k$ from applying a perturbation result in Lemma~\ref{lemma:pertubation} to the dynamic we obtained in the second step. Finally, we prove that at the $d$-th iteration, the estimation is already sufficiently good, and Lemma~\ref{lemma:reduce}, a variant of  \cite[Theorem 2]{Waldspurger2016}, will be used to prove that the algorithm succeeds.

\begin{lemma}\label{lemma:approximate_basis}
There exists $\{\bv_i\}_{i=0}^d$ such that ${\bv}_i$ are i.i.d. sampled from $CN(0,\bI/m)$, $\bu_0=\bv_0/\|\bv_0\|$, and
\begin{align}\label{eq::approximate_basis1}&\text{$\Pr\left(\|\bu_k-{\bv}_k\|>{\frac{\log m}{\sqrt{m}}}\right)<C\exp(-C\log^2 m)$ for $k=0,1$}\\ &\text{$\Pr\left(\|\bu_k-{\bv}_k\|>2\sqrt{\frac{n}{m}}\right)<C\exp(-Cn)$ for $2\leq k\leq  d$}
\label{eq::approximate_basis2}
\\ &\text{$\Pr\left(|\tilde{c}_{k}^{(k)}|>2\sqrt{\frac{n}{m}}\right)<C\exp(-Cn)$ for $1\leq k\leq  d$}.\label{eq::approximate_basis3}
\end{align}
In addition, we have the following properties:
\begin{equation}\label{eq:additional_assumptions}
\|\bv_i\|\leq 2, \|\bv_i\|_\infty\leq \frac{\log m}{\sqrt{m}}, \,\,\text{for all $0\leq i\leq d$}
\end{equation}
holds with probability $1-2m(d+1)\exp(-\log^2m)-2m(d+1)\exp(-\log^2m/4)$.
\end{lemma}

\begin{lemma}\label{lemma:pertubation}
For any  $\bx\in\bbC^m$ and $\bx\sim CN(0,\bI_{m\times m}/m)$,  with probability at least $1-m\exp(-n/6)$, we have
\[
\frac{1}{m}\|\phase(\bx+\by)-\phase(\bx)\|_1\leq C\log m \max\left(\|\by\|,\frac{n}{m}\right)
\]
\end{lemma}
\begin{lemma}\label{lemma:iteration_v}
For any  $\bx\in\bbC^m$ defined by $\bx=\sum_{i=0}^{d}c_i\bv_i$, where ${\sum_{i=0}^{d}c_i^2}=1$ and $\bv_i\sim CN(0,\bI/m)$ for all $0\leq i\leq d$. Define $f,g:\reals\rightarrow\reals$ by
\[f(c)=\frac{1}{c}\Expect_{x_0,x_1\sim CN(0,1)}\frac{cx_0+\sqrt{1-c^2}x_1}{|cx_0+\sqrt{1-c^2}x_1|}|x_0|x_0^*\] and \[g(c)=\frac{1}{\sqrt{1-c^2}}\Expect_{x_0,x_1\sim CN(0,1)} \frac{cx_0+\sqrt{1-c^2}x_1}{|cx_0+\sqrt{1-c^2}x_1|}|x_0|x_1^*,\]
  then \begin{align*}\Pr\left(\!\!|\bv_0^*[\bx\odot \bv_0] \!\!-\!\! f(|c_0|)c_0|\!\!<\frac{\log^2 m}{\sqrt{m}}\!\!\right)&>\!\!1-\exp(-C\log^4m),\end{align*}
  and for any $1\leq j\leq d$,
  \begin{align*}
\Pr\left(\left|\bv_j^*[\bx\odot \bv_0] - g(|c_0|){c_j}\right|<\sqrt{\frac{n}{{m}}}\right)&>1-\exp(-Cn).
\end{align*}
\end{lemma}

\begin{lemma}\label{lemma:expectation}
Given any $0<C_0<1$, there exists $C_f>1, C_g>0$ depending on $C_0$ such that
\begin{equation}\label{eq:expectation1}
\min_{0<c<C_0}f(c)>C_f,  \min_{0<c<C_0}g(c)>C_g.
\end{equation}
In addition, we have
\begin{equation}\label{eq:expectation2}
\frac{f(c)}{g(c)}\geq 1,\,\,\text{for all $0<c<C_0$.}
\end{equation}
\end{lemma}
The following lemma is a result of  \cite[Theorem 2]{Waldspurger2016}:
\begin{lemma}\label{lemma:reduce}
There exists $0<C_0<1$, $C_1',C_2'>0$ such that if $|c^{(k_0)}_0|>C_0$ for some $k_0>0$, then the algorithm \eqref{eq:alternate_minimization2} converges to the solution with probability $1-\exp(-n/2)-C_1'\exp(-C_2'm)$, in the sense that
\[
\lim_{k\rightarrow\infty}\inf_{\psi\in\reals}\|e^{i\psi}\bz-\bx^{(k)}\|=0.
\]
\end{lemma}

\begin{lemma}\label{lemma:init}
With probability at least $1-1/\log n-\exp(-Cn)$, $|c^{(1)}_0|\leq \frac{1}{2\sqrt{n\log n}}$.
\end{lemma}

For the rest of the proof, we first assume that for all $1\leq k\leq d$, $|c^{(k)}_0|<C_0$, since otherwise Lemma~\ref{lemma:reduce} already implies Theorem~\ref{thm:main}.  The goal is to show that under this assumption, we will have $|c^{(d+1)}_0|>C_0$, and then Lemma~\ref{lemma:reduce} implies Theorem~\ref{thm:main}.

Let $\bc=\{c_i\}_{i=0}^d\in\bbC^{d+1}$, we choose a set of covering balls of radius $n/m$ in the set $\mathcal{S}=\{\bc\in\bbC^{d+1}: \|\bc\|=1, |c_0|\leq C_0\}$. That is, we find a subset $\mathcal{S}_0\subset \mathcal{S}$ such that for any $\bc\in \mathcal{S}$, there exists an element $\bar{\bc}=\{\bar{c}_i\}_{i=0}^d\in\mathcal{S}_0$ such that $\|\bc-\bar{\bc}\|\leq n/m$. Following~\cite[Lemma 5.2]{vershynin2010introduction}, $\mathcal{S}_0$ can be chosen such that $|\mathcal{S}_0|\leq (1+\frac{2m}{n})^{2(d+1)}$.  We assume that for all $\bar{\bc}\in\mathcal{S}_0$, the property in Lemma~\ref{lemma:pertubation} holds for $\bx=\sum_{i=0}^d\bar{c}_i\bv_i$, and the property in Lemma~\ref{lemma:iteration_v} also holds.  Then for all $j=0,1,\cdots,d$, we have
\begin{align}\label{eq:pertubation}
&\left|\bu_j^*\left[\sum_{i=0}^{d}c_i\bu_i\odot \bu_0\right]-\bv_j^*\left[\sum_{i=0}^{d}\bar{c}_i\bv_i\odot \bv_0\right]\right|\\\nonumber
=&\left|\frac{1}{\|\bv_0\|}\bu_j^*\left[\sum_{i=0}^{d}c_i\bu_i\odot \bv_0\right]-\bv_j^*\left[\sum_{i=0}^{d}\bar{c}_i\bv_i\odot \bv_0\right]\right|\\\nonumber
=&\Bigg|\left(\frac{1}{\|\bv_0\|}\bu_j^*-\bv_j^*\right)\left[\sum_{i=0}^{d}c_i\bu_i\odot \bv_0\right]\\\nonumber&-\bv_j^*\left[\left[\sum_{i=0}^{d}c_i\bu_i\odot \bv_0\right]-\left[\sum_{i=0}^{d}\bar{c}_i\bv_i\odot \bv_0\right]\right]\Bigg|\\\nonumber
\leq &\left\|\frac{1}{\|\bv_0\|}\bu_j^*-\bv_j^*\right\|\|\bv_0\|\\\nonumber&+\|\bv_j\|_{\infty}\left\|\mathrm{sign}(\sum_{i=0}^d c_i\bu_i)-\mathrm{sign}(\sum_{i=0}^d \bar{c}_i\bv_i)\right\|_1\|\bv_0\|_\infty\\\nonumber
\leq& \|\bu_j-\bv_j\|+\left|\|\bv_0\|-1\right|\|\bv_j\|\\\nonumber&+\|\bv_j\|_{\infty}\|\bv_0\|_\infty m\max\left(\sum_{i=0}^d c_i\|\bu_i-\bv_i\|+|{c}_i-\bar{c}_i|\|\bv_i\|,\frac{n}{m}\right)\\\nonumber
\leq &2\frac{\log m}{\sqrt{m}} + \log^2m\left(2d\frac{n}{m}+d\max_{2\leq i\leq d}2|c_i|\sqrt{\frac{n}{m}}\right)+\|\bu_j-\bv_j\|.
\end{align}
By the definition of $\mathcal{S}_0$, we have that for each $1\leq k\leq d$, there exists $\bar{\bc}^{(k)}=[\bar{c}_0^{(k)},\cdots, \bar{c}_d^{(k)}]\in\mathcal{S}_0$ such that
\[
\sum_{i=0}^k|\bar{c}_i^{(k)}-{c}_i^{(k)}|^2\leq \left(\frac{n}{m}\right)^2,\,\,\, |\bar{c}_0^{(k)}|<C_0.
\]
Combining the analysis in \eqref{eq:pertubation} (with $\bc, \bar{\bc}$ replaced by $\bc^{(k)}, \bar{\bc}^{(k)}$), and applying   \eqref{eq:alternate_minimization3} and Lemma~\ref{lemma:expectation}, we have that for $1\leq j\leq k$,
\begin{align}\label{eq:induction1}
&\left|\tilde{c}^{(k+1)}_j-g(|\bar{c}_0^{(k)}|)\bar{c}_j^{(k)}\right|\leq
4\sqrt{\frac{n}{m}}\\&+\log^2m\left(2d\max_{2\leq i\leq k}|c_i^{(k)}|\sqrt{\frac{n}{m}}+2d\frac{n}{m}\right)\nonumber
\end{align}
 and for $j=0$,
\begin{align}\label{eq:induction2}
&\left|\tilde{c}^{(k+1)}_0-f(|\bar{c}_0^{(k)}|)\bar{c}_0^{(k)}\right|\leq
3\frac{\log m}{\sqrt{m}}\\&+\log^2m\left(2d\max_{2\leq i\leq k}|c_i^{(k)}|\sqrt{\frac{n}{m}}+2d\frac{n}{m}\right).\nonumber
\end{align}
Combining \eqref{eq:induction1} and \eqref{eq:induction2} with \eqref{eq:alternate_minimization6},
\begin{align}\label{eq:induction3}
&\sqrt{\sum_{i=0}^{k+1}|\tilde{c}^{(k+1)}_i|^2}\geq \sqrt{\sum_{i=0}^{k}|\tilde{c}^{(k+1)}_i|^2}\\\nonumber\geq &\sqrt{f^2(|\bar{c}_0^{(k)}|)|\bar{c}_0^{(k)}|^2+g^2(|\bar{c}_0^{(k)}|)(1-|\bar{c}_0^{(k)}|^2)}-4(k+1)\sqrt{\frac{n}{m}}\\\nonumber&-2d(k+1)\log^2m\left(\max_{2\leq i\leq k}|c_i^{(k)}|\sqrt{\frac{n}{m}}+\frac{n}{m}\right)\\\nonumber
\geq & g(|\bar{c}_0^{(k)}|)-4(k+1)\sqrt{\frac{n}{m}}\\\nonumber&-2d(k+1)\log^2m\left(\max_{2\leq i\leq k}|c_i^{(k)}|\sqrt{\frac{n}{m}}+\frac{n}{m}\right).
\end{align}
Combining \eqref{eq:induction1} and \eqref{eq:induction3} with the update formula \eqref{eq:alternate_minimization5}, the estimation \eqref{eq::approximate_basis3}, and \eqref{eq:expectation2}, using induction  we can verify that for sufficiently large $n,m$, we have
\begin{equation}\label{eq:estc2}
\max_{2\leq j\leq k+1}|c^{(k+1)}_j|<\frac{4}{C_g}\sqrt{\frac{n}{m}}
\end{equation}
for all $0\leq k\leq d-1$. By the assumption ${m}/{\log^3m}\geq Mn^{3/2}\log^{1/2}n$ and Lemma~\ref{lemma:init}, we have
\begin{equation}|c^{(0)}_0|>\frac{Cn\log^3m}{m}.\label{eq:require_init}\end{equation}
Combining it  with \eqref{eq:induction2} and \eqref{eq:alternate_minimization6}, it can be verified by induction that when $M$ is sufficiently large,  
\begin{align*}&
|c^{(k+1)}_0|\geq |\tilde{c}^{(k+1)}_0|\geq \frac{C_f+1}{2} |\bar{c}^{(k)}_0|\geq \frac{C_f+1}{2} (|{c}^{(k)}_0|-\frac{n}{m})
\\\geq &\frac{C_f+3}{4} |{c}^{(k)}_0|.
\end{align*}
for all $1\leq 1\leq k$.

As a result, combining it with Lemma~\ref{lemma:init} we have \begin{equation}\label{eq:cd}
|c^{(d+1)}_0|\geq \left(\frac{C_f+3}{4}\right)^{C_d\log n}\frac{1}{2\sqrt{n\log n}}>C_0
\end{equation}
In fact, the second inequality requires
\[
C_d>\frac{\log C_0 + \frac{1}{2}\log n+\log\log n}{\log(\frac{C_f+3}{4})\log n},
\]
and for large $n$, our choice of $C_d=\frac{1}{2\log(\frac{C_f+3}{4})}+1$ would suffice. With \eqref{eq:cd},  Lemma~\ref{lemma:reduce} proves Theorem~\ref{thm:main}.

In the end, we summarize the probability that the above analysis holds: the proof requires the events in all lemmas, and in addition, the events in Lemma~\ref{lemma:pertubation} and Lemma~\ref{lemma:iteration_v} should hold for all $\bx=\sum_{i=0}^d\bar{c}_i\bv_i$, where $\bar{\bc}=[\bar{c}_0,\bar{c}_1,\cdots,\bar{c}_d]$ is an element in $\mathcal{S}_0$. As a result, the probability is at least
\begin{align*}
&1-2Cmd\exp(-\log^2m)-\exp(-n/2)-C_1'\exp(-C_2'm)\\&-\log n-\exp(Cn)-\Big(1+\frac{2m}{n}\Big)^{2d+2}\Big(m\exp(-n/6)\\&-\exp(-C\log^4m)-d\exp(-Cn)\Big),
\end{align*}
which can be verified to converge to $1$ as $n,m\rightarrow\infty$.
\subsection{Proof of the Main Lemmas}\label{sec:lemma}
\begin{proof}[Proof of Lemma~\ref{lemma:approximate_basis}]
Since $L$ is a random $n$-dimensional subspace in $\bbC^m$, and $\Proj_L$ is a random projection matrix to $L$, $\bu_0$ is random unit vector in $\bbC^m$ that is uniformly sampled from the sphere in $\bbC^m$. Therefore, it can be obtained through $\bv_0\sim CN(0,\bI/m)$ by
\[
\bu_0=\frac{\bv_0}{\|\bv_0\|}.
\]
Applying Lemma~\ref{lemma:norm_gaussian} (with a scaling of $\sqrt{m}$) and $\|\bu_0-\bv_0\|=|\|\bv_0\|-1|$, we proved \eqref{eq::approximate_basis1} for $k=0$.

 Under the $\sigma$-algebra generated by $\bu_0$, the conditional distribution of $L$ is a random subspace generated by
 \[
 \Sp(\bu_0)\oplus L_0,
 \]
where $L_0$ is a random $n-1$-dimensional subspace in the $m-1$-dimensional hyperplane $\Sp(\bu_0)^\perp$ (here $\oplus$ represents the direct sum of two subspaces). Since $\bw^{(1)}$ is a random initialization on $L$ and $\bu^{(1)}$ is the projection of $\bw^{(1)}$ onto $\Sp(\bu_0)^\perp$, $\bu_1$ is a random unit vector on $L_0$. Combining it with the conditional distribution of $L_0$, $\bu_1$ is a random unit vector that is orthogonal to $\bu_0$. As a result, it can be generated from $\bv_1\sim CN(0,\bI/m)$ as follows:
\[
\bu_1=\frac{\bv_1-\bu_0\bu_0^*\bv_1}{\|\bv_1-\bu_0\bu_0^*\bv_1\|}.
\]
Since
\begin{equation}\label{eq:tt1}
Pr(|1-\|\bv_1\||>\frac{\log m}{3\sqrt{m}})<\exp(-C\log^2 m)
\end{equation}
and $\bu_0^*\bv_1\sim CN(0,1/m)$ and as a result, Lemma~\ref{lemma:tail} with $m=1$ implies that \begin{equation}\label{eq:tt2} \Pr\left(|\bu_0^*\bv_1|>\frac{\log m}{3\sqrt{m}}\right)<\exp(-C\log^2m).\end{equation}
Applying Lemma~\ref{lemma:pertubation_uni}, under the event of \eqref{eq:tt2},
\[
\left\|\bu_1-\frac{\bv_1}{\|\bv_1\|}\right\|\leq 2\|\bu_0\bu_0^*\bv_1\|\leq \frac{2\log m}{3\sqrt{m}}
\]
combining it with \eqref{eq:tt1}, \eqref{eq::approximate_basis1} with $k=1$ holds.

To prove \eqref{eq::approximate_basis2}, we first investigate the  conditional distribution of $L$ under the $\sigma$-algebra generated by the algorithm so far, that is, generated by $\{\bu_i\}_{i=0}^{k-1}$ and $\{\bw_i\}_{i=0}^{k-1}$. That is, what is the conditional distribution of $L$ when $\{\bu_i\}_{i=0}^{k-1}$ and $\{\bw_i\}_{i=0}^{k-1}$ are fixed?  Under this $\sigma$-algebra, $L$ satisfies the following properties:
\begin{align*}
\bu_i \in L, \,\,\,\,\,\,&\text{ $0\leq i\leq k-1$} \\ [\bw^{(i)}\odot\by] - \bw^{(i+1)}\bw^{(i+1)*}[\bw^{(i)}\odot\by] \perp L, &\text{ $1\leq i\leq k-2$}.
\end{align*}
The second property above holds since $\bw^{(i+1)}$ is the normalization projection of $\bw^{(i)}\odot\by$ onto $L$, and as a result, $\bw^{(i+1)}\bw^{(i+1)*}[\bw^{(i)}\odot\by]=P_L[\bw^{(i)}\odot\by]$.
Recall that $L$ is a random $n$-dimensional subspace in $\bbC^m$, with this $\sigma$-algebra, its conditional distribution then can be written as
\[
L= \Sp\{\bu_i\}_{i=0}^{k-1}\oplus L_k,
\]
where $L_k$ is a random $n-k$-dimensional subspace in the $m-2k+2$-space $R_k$ that is orthogonal to $\bu_i$, $0\leq i\leq k-1$ and $[\bw^{(i)}\odot\by] - \bw^{(i+1)}\bw^{(i+1)*}[\bw^{(i)}\odot\by]$, $1\leq i\leq k-2$.

 Since $\bw^{(k)}$ is the projection of $\bw^{(k-1)}\odot \by$ onto the subspace $L$ and $\bu_k$ is the unit vector of the projection of $\bw^{(k)}$ to the subspace orthogonal to $\Sp\{\bu_i\}_{i=0}^{k-1}$, in conclusion,  $\bu^{(k)}$ is the unit vector that corresponds to the projection of  $P_{R_k}[\bw^{(k-1)}\odot \by]$ onto $L_k$, a random $n-k$-dimensional subspace in $R_k$. Applying Lemma~\ref{lemma:projection0} (with $m, n, \bbC^m$ replaced by $m-2k+2$, $n-k$, $R_k$), $\bu_k$ can be written as
\begin{equation}\label{eq:uvk}
\bu_k=\sqrt{1-a^2}\bv_k'+a\frac{P_{R_k}[\bw^{(k-1)}\odot \by]}{\|P_{R_k}[\bw^{(k-1)}\odot \by]\|},
\end{equation}
where $\bv_k'$ is a unit vector on $R_k'$, the $m-2k+1$-dimensional subspace inside $R_k$ and orthogonal to $P_{R_k}[\bw^{(k-1)}\odot \by]$, and $a$ is the length of the projection of $P_{R_k}[\bw^{(k-1)}\odot \by]/\|P_{R_k}[\bw^{(k-1)}\odot \by]\|$ onto $L_k$.

Since $L_k$ is a random subspace in $R_k$, $\bv_k'$ is a random unit vector on $R_k'$ and can be derived through $\bv_k\sim CN(0,\bI/m)$ by
\[
\bv_k'=\frac{P_{R_k'}\bv_k}{\|P_{R_k'}\bv_k\|}.
\]
Again use the fact that \[
 \Pr(|1-\|\bv_k\||>\frac{1}{6}\sqrt{\frac{n}{m}})<\exp(-Cn)\]
and Lemma~\ref{lemma:tail} implies
 \[
\Pr({\|\bv_k-P_{R_k'}\bv_k\|}>\frac{1}{6}\sqrt{\frac{n}{m}})<2(2k-1)\exp(-Cn/4(2k-1)),
\]
and Lemma~\ref{lemma:projection} implies
\begin{equation}\label{eq:length_ck}
\Pr(a^2>2\cdot\frac{n-k}{m-2k+2})<4\exp\left(-C(n-k)\right).
\end{equation}
Combining all estimations above  with \eqref{eq:uvk} and Lemma~\ref{lemma:pertubation_uni},  \eqref{eq::approximate_basis2} is proved as follows:
\begin{align*}
&\|\bu_k-\bv_k\|\leq a+\sqrt{1-a^2}\|\bv_k'-\bv_k\|\\
\leq &  a+\sqrt{1-a^2}(|\|\bv_k\|-1|+\|\bv_k'-\frac{\bv_k}{\|\bv_k\|}\|)\\
\leq &  a+(|\|\bv_k\|-1|+2\|P_{R_k'}{\bv_k}-{\bv_k}\|/\|\bv_k\|)\\
\leq & \sqrt{2\cdot\frac{n-k}{m-2k+2}}+\left(\frac{1}{6}\sqrt{\frac{n}{m}}+2\frac{\frac{1}{6}\sqrt{\frac{n}{m}}}{1-\frac{1}{6}\sqrt{\frac{n}{m}}}\right)\leq 2\sqrt{\frac{n}{m}}.
\end{align*}

Note $\tilde{c}_{k}^{(k)}$ is the length of projection of $P_{R_k}[\bw^{(k-1)}\odot \by]$ onto $L_k$, and $\|P_{R_k}[\bw^{(k-1)}\odot \by]\|\leq \|[\bw^{(k-1)}\odot \by]\|=\|\by\|=1$, by the definition of $a$ we have $\tilde{c}_{k}^{(k)}\leq a$. Then \eqref{eq:length_ck} implies \eqref{eq::approximate_basis3}.

At last, \eqref{eq:additional_assumptions}  is obtained by applying Lemma~\ref{lemma:tail} (with union bound and $m=1$ for the $\|\cdot\|_\infty$ norm).
\end{proof}

\begin{proof}[Proof of Lemma~\ref{lemma:pertubation}]
It is based on a combination of Lemma~\ref{lemma:pertubation_sum}, \ref{lemma:order}, and \ref{lemma:pertubation_uni}. In particular, $t=\max(\|\by\|,n/m)$ in Lemma~\ref{lemma:pertubation_sum} (remark: there is a scaling factor of $\sqrt{m}$ between $\bx$ in Lemma~\ref{lemma:pertubation} and Lemma~\ref{lemma:pertubation_sum}).
\end{proof}
\begin{proof}[Proof of Lemma~\ref{lemma:iteration_v}]
The proof is based on two components: first, we have \begin{align}\label{eq:expectation21}\Expect \bv_0^*[\bx\odot \bv_0]=f(|c_0|)c_0, \end{align}
  and for any $1\leq j\leq d$,
  \begin{align}\label{eq:expectation22}
\Expect \bv_j^*[\bx\odot \bv_0] = g(|c_0|){c_j}.
\end{align}
In fact, \eqref{eq:expectation21} follows directly from the definition of $f$, and \eqref{eq:expectation22} follows from the definition of $g$ as follows: that if we write $\bx=c_0\bv_0+\sqrt{1-|c_0|^2}\bv'$ with $\bv'=\frac{1}{\sqrt{\sum_{i=1}^d|c_i|^2}}\sum_{i=1}^dc_i\bv_i$, then $\bv'\sim CN(0,\bI/m)$ and the definition of $g$ implies \begin{align*}\Expect \bv'^*[\bx\odot \bv_0]=\sqrt{1-|c_0|^2}g(|c_0|)\end{align*}
Noting that the correlation between $\bv_j$ and $\bv'$ is  $\frac{c_j}{\sqrt{1-|c_0|^2}}$, \eqref{eq:expectation22} is proved.

Since each element of $\bv_0$ is sampled from $CN(0,1/m)$, it can be verified that the real component and the imaginary component of each entry of $\bx\odot \bv_0$ is sub-gaussian, with sub-gaussian parameter bounded above by $C/\sqrt{m}$. Applying Lemma~\ref{lemma:subgaussian}, this Lemma is proved.

\end{proof}
\begin{proof}[Proof of Lemma~\ref{lemma:expectation}]
We first let $f_0(c)=cf(c)$ and $g_0(c)=\sqrt{1-c^2}g(c)$, and show the following arguments:
\begin{align}\label{argument1}
&\text{$f_0$ and $g_0$ are bounded for $0\leq c\leq C_0$}.\\\label{argument2}
&\text{$f_0$, $f_0'$, and $g_0$ are Lipschitz continuous for $0\leq c\leq C_0$}.\\
&\text{$f_0'(0)>1$}\label{argument3}.
\end{align}

Now we verify \eqref{argument1}-\eqref{argument3}. First, by the rotational invariance of the complex normal distribution, $f_0$ can be equivalently defined by
\[
f_0(c)=\Expect_{x_0,x_1\sim CN(0,1)}\frac{x_0}{|x_0|}|cx_0+\sqrt{1-c^2}x_1|(cx_0+\sqrt{1-c^2}x_1)^*.
\]
It is clear that $f_0(c)=\Expect_{x_0,x_1\sim CN(0,1)}<\Expect (|x_0|+|x_1|)^2$ and is therefore bounded.

For all $0\leq c, c'\leq 1$,
\[
|cx_0+\sqrt{1-c^2}x_1-(c'x_0+\sqrt{1-c'^2}x_1)|\leq |c-c'||x_0|+\frac{|c-c'|}{1-C_0^2}|x_1|,
\]
so we have
\begin{align*}
|f(c)-f(c')|\leq |c-c'| \Expect 2\left(|x_0|+\frac{1}{1-C_0^2}|x_1|\right)(|x_0|+|x_1|)
\end{align*}
and as a result,  $f$ is Lipschitz continuous for $0\leq c\leq C_0$ 

$f'$ can be written as
\begin{align*}
&f'(c)=\Expect \frac{x_0}{|x_0|}\Bigg(\left({(\sqrt{1-c^2}-\frac{c^2}{\sqrt{1-c^2}})x_0^*x_1+c|x_0|^2-c|x_1|^2}\right)\times\\&\frac{(cx_0+\sqrt{1-c^2}x_1)^*}{{|cx_0+\sqrt{1-c^2}x_1|}}+|cx_0+\sqrt{1-c^2}x_1|(x_0-\frac{c}{\sqrt{1-c^2}}x_1)^*\Bigg)
\end{align*}
By Lemma~\ref{lemma:pertubation_uni},
\begin{align*}
&\left|\frac{cx_0+\sqrt{1-c^2}x_1}{|cx_0+\sqrt{1-c^2}x_1|}-\frac{c'x_0+\sqrt{1-c'^2}x_1}{|c'x_0+\sqrt{1-c'^2}x_1|}\right|\\\leq &2\frac{|(c-c')x_0+(\sqrt{1-c^2}-\sqrt{1-c'^2})x_1|}{|cx_0+\sqrt{1-c^2}x_1|}\\\leq &2|c-c'|\frac{|x_0|+\frac{1}{\sqrt{1-C_0^2}}|x_1|}{|cx_0+\sqrt{1-c^2}x_1|}.
\end{align*}
Considering that $\Expect |x_0|^{k_1}|x_1|^{k_2}/|cx_0+\sqrt{1-c^2}x_1|$ exists for all $k_1,k_2\geq 0$ (note that $\Pr(|cx_0+\sqrt{1-c^2}x_1|<r)<r^2$ as shown in Lemma~\ref{lemma:complexgaussian}), one can show that $f'$ is Lipschitz continuous for $0\leq c\leq C_0$.

Similarly, one can prove that $g$ is bounded for $0\leq c\leq 1$and Lipschitz continuous for $0\leq c\leq C_0$.
Third,
\[
f_0'(0)=\Expect \frac{x_0}{|x_0|}\left(x_0^*|x_1|+|x_1|x_0^*\right)=2\Expect |x_0||x_1|\approx 1.56>1.
\]

Next, we will prove Lemma~\ref{lemma:expectation} based on \eqref{argument1}-\eqref{argument3}.
 To prove that $\min_{0<c<C_0}g(c)>C_g$, we first note that $g(c)$ is Lipschitz continuous since $g_0(c)$ and $\frac{1}{\sqrt{1-c^2}}$ are Lipschitz continuous. Then this can be verified numerically by calculating $g(c)$ for a sufficiently dense sampling of points in $0<c<C_0$. For example, if  the Lipschitz constant of $g$ is $L_{g}$, and for $c=\{kC_g/L_{g}\}_{k=0}^{\lfloor\frac{C_0L_{g}}{C_g}\rfloor}$, $g(c)>2C_g$, then we have $\min_{0<c<C_0}g(c)>C_g$. As shown in Figure~\ref{fig:expectation}, such numerical verification is doable.

To prove  $\min_{0<c<C_0}f(c)>C_f$ for some $C_f>1$, one first note that since $f(c)=\frac{f_0(c)-f_0(0)}{c-0}=\int_{x=0}^cf_0'(x)$, $f_0'(0)>1$ and the Lipschitz continuity of $f_0'$ (assuming that the parameter is $L_{f_0'}$), it shows that
\[
\min_{0<c<\frac{f_0'(0)-1}{2L_{f_0'}}}f(c)>\frac{f_0'(0)+1}{2}.
\]
It remains the show that $\min_{\frac{f_0'(0)-1}{2L_{f_0'}}<c<C_0}f(c)>C_f$ for some $C_f>1$. The proof is then similar to the proof of $g$.

To prove \eqref{eq:expectation2}, one may first use the same technique to verify that $\max_{0<c<C_0}g(c)<1$. Combining it with \eqref{eq:expectation1}, \eqref{eq:expectation2} is proved.

We included the numerical values of $f(c)$, $g(c)$, and $f(c)/g(c)$ in Figure~\ref{fig:expectation}, to show that the inequalities \eqref{eq:expectation1} and \eqref{eq:expectation2} hold empirically, and as a result, the numerical verification method above would work.
\begin{figure}\label{fig:expectation}
  \centering
    \includegraphics[width=0.5\textwidth]{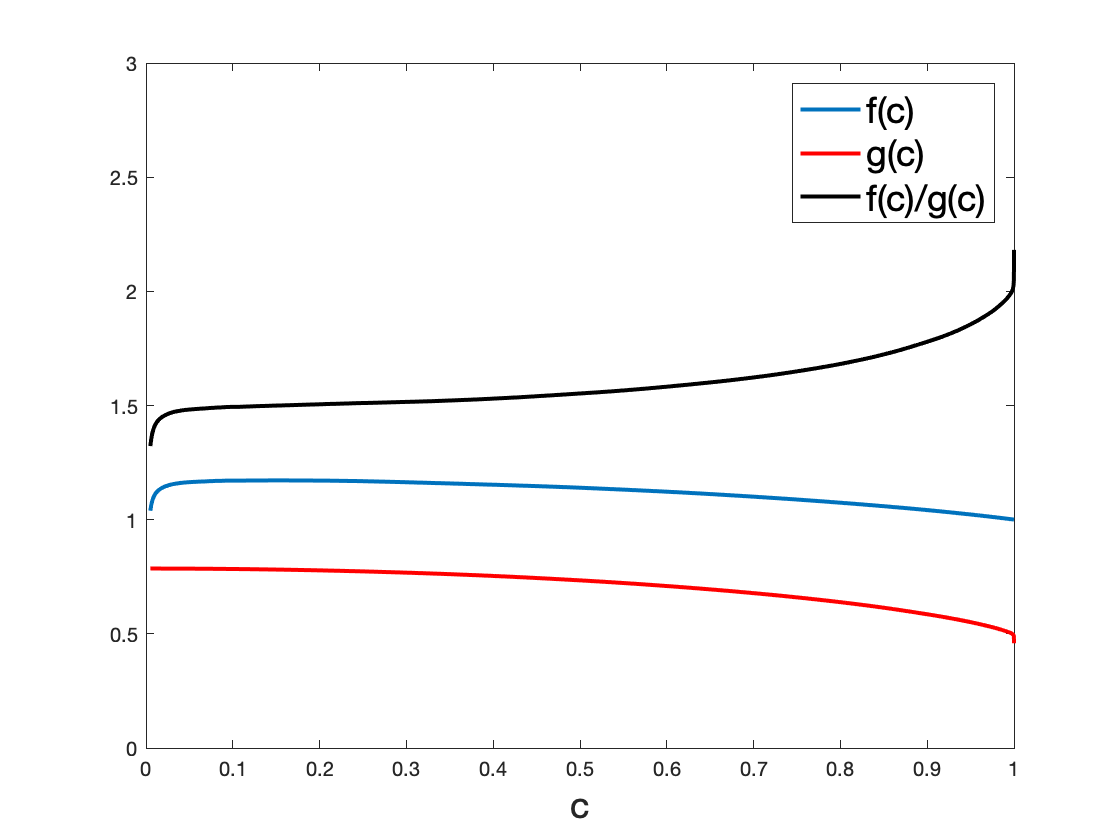}
      \caption{The empirical values of $f$, $g$, and $f/g$.}
\end{figure}
\end{proof}

\begin{proof}[Proof of Lemma~\ref{lemma:reduce}]
For convenience, we first write down \cite[Theorem 2]{Waldspurger2016} explicitly:
\begin{thm}[\cite{Waldspurger2016}, Theorem 2]
There exists $C_0', C_1', C_2', M>0$ such that when $m>Mn$, then with probability at least $1-C_1'\exp(-C_2'm)$, for any $\bx\in\bbC^n$ such that
\[
\inf_{\psi\in\reals}\|e^{i\psi}\bz-\bx\|\leq C_0'\|\bz\|,
\]
then
\[
\inf_{\psi\in\reals}\|e^{i\psi}\bz-\bx^+\|\leq \delta\inf_{\psi\in\reals}\|e^{i\psi}\bz-\bx\|,
\]
where $\bx^+$ is the vector obtained by applying one iteration of the standard alternating projection algorithm (without normalization) \eqref{eq:alternating_algorithm} to $\bx$, and with $\{\ba_i\}_{i=1}^m$ i.i.d. sampled from $CN(0,\bI)$.
\end{thm}

By the analysis in Section~\ref{sec:equivalent}, \eqref{eq:alternating_algorithm} is equivalent to the algorithm that we are analyzing in \eqref{eq:alternate_minimization2} in terms of $\bw^{(k)}$ as in $\bbC^m$. Therefore, \cite[Theorem 2]{Waldspurger2016} implies that for $\bA^+=\bA^*(\bA^*\bA)^{-1}$, if
\[
\inf_{\psi\in\reals}\|\bA^+(e^{i\psi}\bu_0-\bw^{(k_0)})\|\leq C_0'\|\bA^+\bu_0\|,
\]
then for any $k\geq k_0$,
\[
\inf_{\psi\in\reals}\|\bA^+(e^{i\psi}\bu_0-\bw^{(k+1)})\|\leq \delta \inf_{\psi\in\reals}\|\bA^+(e^{i\psi}\bu_0-\bw^{(k)})\|.
\]
Since $\bA$ is a complex Gaussian, \cite[Theorem 2.13]{Davidson2001} implies that the condition number of  $\bA$ is bounded with high probability:
\[
\Pr\Big(
\frac{\sigma_{\max}(\bA)}{\sigma_{\min}(\bA)}\leq \frac{\sqrt{m}+\sqrt{n}+t}{\sqrt{m}-\sqrt{n}-t}
\Big)\leq 1-2\exp(-t^2/2).
\]
Combining it (use $t=\sqrt{n}$) with $\inf_{\psi\in\reals}\|e^{i\psi}\bu_0-\bw^{(k)}\|=\sqrt{(1-|c_0^{(k)}|)^2+\sum_{i=1}^k|c_i^{(k)}|^2}=\sqrt{(1-|c_0^{(k)}|)^2+1-|c_0^{(k)}|^2}$, there exists $0<C_0<1$ such that when $|c_0^{(k)}|>C_0$,  $\inf_{\psi\in\reals}\|\bA^+(e^{i\psi}\bu_0-\bw^{(k)})\|<C_0'\|\bA^+\bu_0\|$ (note that $\bA$ and $\bA^+$ have the same condition numbers), and then \cite[Theorem 2]{Waldspurger2016} implies $\lim_{k\rightarrow\infty}\inf_{\psi\in\reals}\|\bA^+(e^{i\psi}\bu_0-\bw^{(k)})\|= 0$. Applying the fact that the condition number of  $\bA^+$ is bounded again, we have $\lim_{k\rightarrow\infty}\inf_{\psi\in\reals}\|e^{i\psi}\bu_0-\bw^{(k)}\|= 0$.
\end{proof}
\begin{proof}[Proof of Lemma~\ref{lemma:init}]
WLOG we may assume that $L=\Sp(\be_1,\cdots,\be_n)$ and $\bu_0=\be_1$, and $\bw^{(1)}\sim CN(0,P_L)$ (since $\bw^{(1)}$ is a random vector on $L$). Then
\[
|c^{(1)}_0|=\frac{|w^{(1)}_1|}{\|\bw^{(1)}\|}.
\]
Applying Lemma~\ref{lemma:norm_gaussian} and note that $\|\bw^{(1)}\|^2$ is the sum of $n$ unit complex gaussian squared, $\Pr(\|\bw^{(1)}\|>2\sqrt{n})<2\exp(-Cn)$; and Lemma~\ref{lemma:complexgaussian} implies that  $\Pr(\|\bw^{(1)}\|<1/\sqrt{\log n})<\log n$.

\end{proof}
\subsection{Axillary Lemmas}\label{sec:auxillary}
\begin{lemma}\label{lemma:projection0}
Assuming that projection of a unit vector $\bx\in\bbC^m$ to a random $n$-dimensional subspace $L$ has length $a$, i.e., $\|P_L\bx\|=a$, then
\[
\frac{P_L\bx}{\|P_L\bx\|}= a\bx+\sqrt{1-a^2}\bv,
\]
where $\bv$ is a unit vector perpendicular to $\bx$, that is, $\bv^*\bx=1$.
\end{lemma}
\begin{proof}
Since $\frac{P_L\bx}{\|P_L\bx\|}$ is a unit vector, we may assume that \begin{equation}\label{eq:assumption_projection}\frac{P_L\bx}{\|P_L\bx\|}= b\bx+\sqrt{1-b^2}\bv,\end{equation} where $\|\bv\|=1$ and $\bv^*\bx=0$. It remains to prove $a=b$.

By the definition of projection, we have $(\bx-P_L\bx)\perp P_L\bx$, i.e., $P_L\bx^*(\bx-P_L\bx)=0$. Plug in the assumption \eqref{eq:assumption_projection} and $\|P_L\bx\|=a$ we have
\[
\left(b\bx+\sqrt{1-b^2}\bv\right)^*\left((1-ab)\bx-a\sqrt{1-b^2}\bv\right)=0.
\]
With $\bv^*\bx=0$ and $\|\bx\|=\|\bv\|=1$, it implies $b(1-ab)=a(1-b^2)$ and $a=b$.
\end{proof}
\begin{lemma}\label{lemma:projection}
Given a  vector $\bx\in\reals^m$ and a random $n$-dimensional subspace $L$, then
\begin{align*}
&\Pr\left(\frac{1-\epsilon}{1+\epsilon}\leq \frac{m\|P_L\bx\|^2}{n\|\bx\|^2}\leq \frac{1+\epsilon}{1-\epsilon}\right)\geq \\&1-4\exp\left(-cn\min\left(\frac{\epsilon^2}{C^2},\frac{\epsilon}{C}\right)\right),
\end{align*}\end{lemma}
\begin{proof}
WLOG we may assume that $\bx\sim CN(0,\bI)$ and $L$ is the subspace spanned by the first $n$ standard basis $\be_1,\cdots,\be_n$. Then $\|\bx\|^2=\sum_{i=1}^m|x_i|^2$ and $\|P_L\bx\|^2=\sum_{i=1}^n|x_i|^2$. Applying Lemma~\ref{lemma:norm_gaussian}, we have
\begin{align*}
&\Pr\left((1-\epsilon)m\leq \sum_{i=1}^m|x_i|^2\leq (1+\epsilon)m\right)
\\\geq& 1-2\exp\left(-cm\min\left(\frac{\epsilon^2}{C^2},\frac{\epsilon}{C}\right)\right)\\
&\Pr\left((1-\epsilon)n\leq \sum_{i=1}^n|x_i|^2\leq (1+\epsilon)n\right)\\\geq &1-2\exp\left(-cn\min\left(\frac{\epsilon^2}{C^2},\frac{\epsilon}{C}\right)\right)
\end{align*}
Combining these two inequalities and $m\geq n$, the lemma is proved.
\end{proof}
\begin{lemma}\label{lemma:complexgaussian}
For $x\sim CN(0,1)$ and any $r>0$, $\Pr(|x|\leq r)<r^2$.
\end{lemma}
\begin{proof}
By the definition of $CN(0,1)$, $\Pr(|x|\leq r)$ is the equivalent to $\Pr(\|\by\|\leq r)$ for $\by\in\reals^2$ and sampled from $N(0,\bI_{2\times 2}/2)$, which has a probability density function of $\frac{1}{\pi}\exp(-\|\by\|^2)$. This function is maximized at $\by=\mathbf{0}$  with a value of $1/\pi$, and as a result, $\Pr(\|\by\|\leq r)<\pi r^2 \cdot \frac{1}{\pi}=r^2$.\end{proof}

\begin{lemma}\label{lemma:pertubation_sum}
Given a vector $\bx\in \reals^m$  and  $\bx\sim CN(0,\bI_{m\times m})$. If  $\by\in \reals^m$ satisfies $\frac{1}{m}\sum_{i=1}^m|y_i|^2\leq t^2$ and $t\geq n/m$, then with probability at least $1-m\exp(-n/6)$, we have
\[
\frac{1}{m}\sum_{i=1}^m\max\left(\frac{|y_i|}{|x_i|},1\right)\leq (4+\sqrt{2}l) t
\]
for $l=\left \lfloor{-\log_2t }\right \rfloor. $
\end{lemma}
\begin{proof}
WLOG we may rearrange the indices and assume that $|x_1|\leq |x_2|\leq \cdots \leq |x_n|$. Then Lemma~\ref{lemma:order} implies that
\begin{equation}\label{eq:xj}
\Pr(|x_j| > \sqrt{j/2m})\geq 1-\exp(-j/6).
\end{equation}
Applying a union bound,
\begin{equation}\label{eq:xj1}
\Pr\left(|x_j| > \sqrt{j/2m}\,\,\text{for all $j\geq n$}\right)\geq 1-m\exp(-n/6).
\end{equation}

When the event in \eqref{eq:xj1} holds, for all $j\geq n$ we have
\begin{equation}\label{eq:interval}
\sum_{i=j+1}^{2j}\frac{|y_i|}{|x_i|}\leq \frac{1}{|x_j|}\sum_{i=j+1}^{2j}|y_i|
\leq \frac{1}{|x_j|}\sqrt{j\sum_{i=j+1}^{2j}{|y_i|^2}}\leq mt\sqrt{2}.
\end{equation}
Combining \eqref{eq:interval} for $j=tm, 2tm, 4tm,\cdots, 2^ltm$ ($l$ is the largest integer such that $2^lt<1$) and $j=m/2$, we have
\begin{equation}\label{eq:interval1}
\sum_{i=tm+1}^{m}\frac{|y_i|}{|x_i|}\leq (2+l) mt\sqrt{2}.
\end{equation}
In addition, it is clear that \begin{equation}\label{eq:interval2}
\sum_{i=1}^{tm}\max\left(\frac{|y_i|}{|x_i|},1\right)\leq tm.\end{equation}
Combining \eqref{eq:interval1} and \eqref{eq:interval2}, the lemma is proved. (integer issue?)
\end{proof}

\begin{lemma}\label{lemma:order}
For a  vector  $\bx\in\bC^m$ and $\bx\sim \sum CN(0,\bI_{m\times m})$, we have
\[
\Pr\left(\sum_{i=1}^mI(|x_i|\leq r)<2r^2m\right)>1-\exp(-r^2m/3).
\]
\end{lemma}
\begin{proof}
We apply Lemma~\ref{lemma:bernoullis} with $p=\Pr(|x_1|\leq r)$ and $\delta=2r^2/p-1$. Applying Lemma~\ref{lemma:complexgaussian},
\[
\delta p=2r^2-\Pr(|x_1|\leq r)>r^2,
\]
which implies Lemma~\ref{lemma:order}.
\end{proof}

\begin{lemma}\label{lemma:pertubation_uni}
For any complex number $x, y\in\bbC$, $|\mathrm{phase}(x+y)-\mathrm{phase}(x)|\leq \min(2\frac{|y|}{|x|},2)$. Similarly, for any vector $\bu,\bv\in\bbC^m$, $\|\frac{\bu+\bv}{\|\bu+\bv\|}-\frac{\bv}{\|\bv\|}\|\leq\min(2\frac{\|\bu\|}{\|\bv\|},2)$.
\end{lemma}
\begin{proof}
WLOG we only need to prove the first sentence and we may assume that $x=1$ and $|y|=r$. Then $\mathrm{phase}(x)=e^{i0}=1$, and on the complex plane, $x+y$ lies on a circle center at $1$ with radius $r$.

\begin{figure}\label{fig:lemma}
  \centering
    \includegraphics[width=0.35\textwidth]{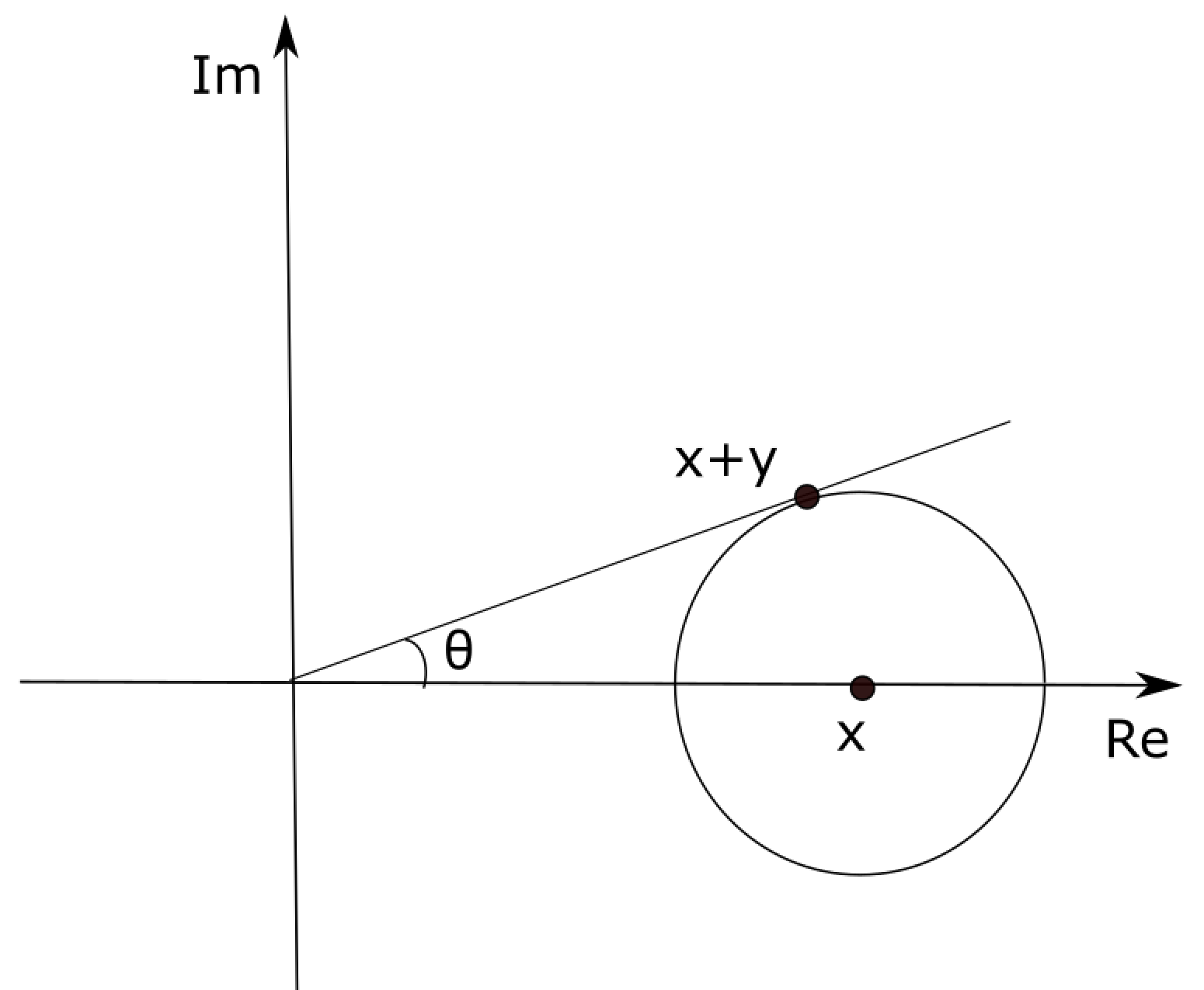}
      \caption{Visualization of the proof of Lemma~\ref{lemma:pertubation_uni} when $r<1$.}
\end{figure}
When $r\geq 1$, $|\mathrm{phase}(x+y)-\mathrm{phase}(x)|$ is maximized when $y=-r$ and $\mathrm{phase}(x+y)=-1$, then we have $|\mathrm{phase}(x+y)-\mathrm{phase}(x)|=2$.

When $r<1$, we would like to find a point on the circle such that its direction is as far from the direction of x-axis as possible. As visualized in Figure~\ref{fig:lemma}, $|\mathrm{phase}(x+y)-\mathrm{phase}(x)|$ is achieved when the line connecting $x+y$ and the origin is tangent to the circle. It implies that the maximal value is $|e^{i\theta}-1|$, where $\theta=\sin^{-1}r$. Then we have the estimation $|e^{i{\theta}}-1|=2\sin(\theta/2)= \sin(\theta)/\cos(\theta/2)\leq \sqrt{2}\sin(\theta)=\sqrt{2}r$ (the inequality uses the fact that $\theta\leq \pi/2$).

Combining these two cases, Lemma~\ref{lemma:pertubation_uni} is proved.
\end{proof}

\begin{lemma}[Sum of sub-gaussian variables, Proposition 5.10 in~\cite{vershynin2010introduction}]\label{lemma:subgaussian}
Given $X_1,\cdots,X_n$ i.i.d. from a distribution with zero mean and sub-gaussian norm defined by $\|X\|_{\psi_2}=\sup_{p\geq 1}p^{-1/2}(\Expect |X|^p)^{1/p}$, then
\[
\Pr\left(\left|\frac{1}{n}\sum_{i=1}^nX_i\right|\geq t\right)\leq  \exp\left(-\frac{cnt^2}{\|X\|_{\psi_2}^2}+1\right)
\]
\end{lemma}
\begin{lemma}[Sum of sub-exponential variables, Corollary 5.17 in~\cite{vershynin2010introduction}]\label{lemma:subexponential}
Given $X_1,\cdots,X_n$ i.i.d. from a distribution with zero mean and sub-exponential norm defined by $\|X\|_{\psi_1}=\sup_{p\geq 1}p^{-1}(\Expect |X|^p)^{1/p}$, then
\[
\Pr\left(\left|\frac{1}{n}\sum_{i=1}^nX_i\right|\geq t\right)\leq  2\exp\left(-cn\min\left(\frac{t^2}{\|X\|_{\psi_1}^2},\frac{t}{\|X\|_{\psi_1}}\right)\right)
\]
\end{lemma}

\begin{lemma}\label{lemma:bernoullis}
$X_1,X_2,\cdots$ are i.i.d. Bernoulli variables with expectation $p$, then for any $\delta>1$,
\[
\Pr(\sum_{i=1}^mX_i>(1+\delta)pm)\leq \exp(-m\delta p/3).
\]
\end{lemma}
\begin{proof}
It follows from \cite[Theorem 4,4]{chernoff} and the observation that when $\delta>1$, $(1+\delta)\log(1+\delta)>\frac{4}{3}\delta$.
\end{proof}

\begin{lemma}\label{lemma:norm_gaussian}
For $\bv\sim CN(0,\bI)$, $\Pr(\frac{1}{m}|\|\bv\|^2-1|>t)<2\exp\left(-cm\min\left(\frac{t^2}{C^2},\frac{t}{C}\right)\right)$
\end{lemma}
\begin{proof}
We remark that $\|\bv\|^2-m=\sum_{i=1}^m((\Re(v_i)^2-1/2)+(\Im(v_i)^2-1/2))$, and both $\Re(v_i)$ and $\Im(v_i)$ are i.i.d. sampled from $N(0,\frac{1}{2})$. Since sub-gaussian squared is sub-exponential~\cite[Lemma 5.14]{vershynin2010introduction} with mean $1/2m$, and after centering, a sub-exponential distribution is still sub-exponential~\cite[Remark 5.18]{vershynin2010introduction}, $\Re(v_i)^2-\frac{1}{2}$ and $\Im(v_i)^2-\frac{1}{2}$ are i.i.d. sampled from a sub-exponential distribution with sub-exponential norm smaller than a constant $C$. Applying Lemma~\ref{lemma:subexponential}, Lemma~\ref{lemma:norm_gaussian} is proved.
\end{proof}
\begin{lemma}\label{lemma:tail}
For any $\bx\sim CN(0,\bI_{m\times m})$, $\Pr(\|\bx\|>t)\leq 2m\exp(-t^2/4m)$.
\end{lemma}
\begin{proof}
It follows from the classic tail bound: $\Pr(|N(0,1)|>t)\leq \exp(-t^2/2)$, and a union bound of all real components of imaginary components of each element of $\bx$, which are i.i.d. sampled from $N(0,1/2)$.
\end{proof}

\section{Discussions}
The current paper justifies the  convergence of alternating minimization algorithm with random initialization for phase
retrieval. Specifically, we demonstrate that it succeeds with $m/\log^3m>Mn^{1.5}\log^{0.5}n$ for some $M>0$. A future direction is to find a better sample complexity, possibly via more sophisticated arguments: empirically, the algorithm succeeds with $m>O(n)$.  It would also be interesting to compare the decoupling approach in this work and the leave-one-out approach in \cite{Chen2018}, both in phase retrieval and in broader settings.
\bibliographystyle{abbrv}
\bibliography{bib-online}

\end{document}

%% file: rrp-macro-file.tex
\usepackage{multirow}

\usepackage{amsthm}
\usepackage{amsmath}
\usepackage{amssymb}
\usepackage{bm}
\usepackage{mathrsfs}
\usepackage[dvips]{graphicx}

\usepackage{algorithmic}
\usepackage{algorithm}

\usepackage{color}

\usepackage{url}

\usepackage{supertabular}

\newtheorem{thm}{Theorem}

\newtheorem{lemma}[thm]{Lemma}

\theoremstyle{definition}

\theoremstyle{remark}

\numberwithin{thm}{section}

\DeclareMathAlphabet{\mathsfsl}{OT1}{cmss}{m}{sl}

\renewcommand{\phi}{\varphi}

\newcommand{\Expect}{\operatorname{\mathbb{E}}}

\newcommand{\mtx}[1]{\bm{#1}}

\newcommand{\range}{\operatorname{range}}

\newcommand{\Proj}{\ensuremath{\mtx{\Pi}}}

\newcommand{\bx}{\mathbf{x}}

\newcommand{\ba}{\mathbf{a}}

\newcommand{\by}{\mathbf{y}}

\newcommand{\bc}{\mathbf{c}}
\newcommand{\bz}{\mathbf{z}}

\newcommand{\bD}{\mathbf{D}}

\newcommand{\bbC}{\mathcal{C}}

\newcommand{\bU}{\mathbf{U}}
\newcommand{\bw}{\mathbf{w}}

\def\reals{\mathbb{R}}

\def\bx{\mathbf{x}}
\def\be{\mathbf{e}}
\def\bu{\mathbf{u}}

\def\b0{\mathbf{0}}

\def\bU{\mathbf{U}}

\def\bC{\mathbf{C}}
\def\bv{\mathbf{v}}
\def\bA{\mathbf{A}}

\def\phase{\mathrm{phase}}

\def\bI{\mathbf{I}}

\def\Sp{\mathrm{Sp}}